\documentclass[12pt]{amsart}
\usepackage[english]{babel}
\usepackage{pgf}
\usepackage[utf8]{inputenc}

\usepackage[centering]{geometry}
\usepackage{mathtools} 
\usepackage[normalem]{ulem} 

\usepackage{array}
\usepackage{tabularx}
\usepackage{graphicx}
\usepackage{amsmath}
\usepackage{amssymb}
\usepackage{amsthm}

\usepackage{enumerate}

\usepackage{tikz}
\usetikzlibrary{arrows}
\usepackage{tkz-euclide}
\usepackage{pgfplots}

\usepackage{color}
\usepackage{xcolor}
\definecolor{myred}{rgb}{0.2,0,0}
\definecolor{myblue}{rgb}{0,0,0.6}
\definecolor{mygreen}{rgb}{0,0.2,0}
\usepackage[colorlinks=true,linkcolor=myred,citecolor=myblue,urlcolor=mygreen]{hyperref}

\numberwithin{equation}{section}

\allowdisplaybreaks

\newcommand{\D}{\mathcal{D}}
\newcommand{\pd}{\mathcal{D}^*}
\renewcommand{\l}{\ell}

\renewcommand{\b}{\beta}
\DeclareMathOperator{\card}{\mathrm{card}}
\renewcommand{\mod}{\operatorname{mod}}

\renewcommand{\P}{\mathbb{P}}

\newcommand{\digitsum}{s}
\newcommand{\tL}{1} 
\newcommand{\tO}{0} 

\newtheorem{theorem}{Theorem}[section]
\newtheorem{lemma}[theorem]{Lemma}

\newtheorem{proposition}[theorem]{Proposition}
\newtheorem{property}[theorem]{Property}
\newtheorem{conjecture}[theorem]{Conjecture}

\newtheorem{claim}[theorem]{Claim}

\theoremstyle{definition}

\theoremstyle{remark}

\usepackage{enumitem}

\begin{document}
 \selectlanguage{english}
  \title{Binary-ternary collisions and the last significant digit of $n!$ in base 12}
 \author[J.-M. Deshouillers, P. Jelinek, L. Spiegelhofer]{ Jean-Marc Deshouillers, Pascal Jelinek, Lukas Spiegelhofer}

\maketitle

\begin{abstract} 
The third-named author recently proved [Israel J. of Math. 258 (2023), 475--502] that there are infinitely many \textit{collisions} of the base-2 and base-3 sum-of-digits functions. In other words, the equation
$$
s_2(n)=s_3(n)
$$
admits infinitely many solutions in natural numbers. We refine this result and prove that every integer $a$ in $\{1, 2, \cdots, 11\}$ appears as the last nonzero digit of $n!$ in base $12$ infinitely often.
\end{abstract}

\section{Introduction}
\label{sec_Intro}
\let\thefootnote\relax\footnote{Classification numbers: 11A63, 11K31, 11B75.\\ \indent Key expressions:  sums of digits in bases 2 and 3, last significant digit of $n!$}
The sequence A136698 of the On-Line Encyclopedia of Integer Sequences \cite{OEIS}  provides the last nonzero digit of the number $n!$ in base 12. \\
Let us define this notion: when we write a positive integer $n$ in base $b$ (where $b$ is an integer larger than $1$) as
\begin{equation}\label{ninbaseb}
n = \sum_{j \ge 0} a_j^{(b)}(n) b^j \; \text{ where } \;  a_j^{(b)}(n) \in \D_b =\{0, 1, \ldots, b-1\},
\end{equation}
its last nonzero digit in base $b$ is defined by
\begin{equation}\label{lnzdigit}
\l_b(n)=\min\{j \colon a_j(n) \neq 0\}.
\end{equation}
A quick glance at the first elements of the above-mentioned sequence reveals that all the numbers in $\pd _{12}= \{1, 2, \ldots, 11\}$ occur in the sequence, the numbers $4$ and $8$ seeming to occur more frequently. \\
Indeed, it was proved by I. Z. Ruzsa and the first-named author in 2011 \cite{DR} that the digits $4$ and $8$ each occur with asymptotic density $1/2$. The next year, the first-named author showed in \cite{D1} that each of the digits $3$, $6$ and $9$ occurs infinitely often. The paper concludes with two open questions: do all the numbers from $\{1, 2, \ldots, 10,  11\}$ occur infinitely often in the sequence $\left(\ell_{12}(n!)\right)_n$? How often do $3, 6$ and $9$ occur in $\left(\ell_{12}(n!)\right)_{n\le N}$ when $N$ is large?
The first aim of this paper is to give answers to these questions.
\begin{theorem}\label{thmlnzdigit}
There exist two  real numbers $\delta$ and $\eta$, with $0 < \delta \le \eta <1$, such that for each integer $a$ in $\{1, 2, 3, 5, 6, 7, 9, 10, 11\}$, we have
\begin{equation}\label{eqlnzdigit}
N^{\delta} \ll \card\{n \le N \colon \l_{12}(n!)=a\} \ll N^{\eta}.
\end{equation}
\end{theorem}

As we shall explain in Section \ref{seclnz} our question is closely connected to the existence of integers $n$ such that the sums of their digits in bases $2$ and $3$ are equal, where the sum of the digits of the integer $n$ in base $b$ is given by 
\begin{equation}\label{sumdig}
s_b(n) = \sum_{j \ge 0} a_j^{(b)}(n), 
\end{equation}
where $a_j^{(b)}(n)$ is introduced in (\ref{ninbaseb}). The third-named author \cite{S} has recently proved that there exist infinitely many integers $n$ such that $s_2(n) = s_3(n)$. Our proof of Theorem \ref{thmlnzdigit} relies on a slight extension of this result.
\begin{theorem}\label{thmJS}
Let $k, \l, a$ be nonnegative integers. There exists a positive real number $\b$ such that 
\begin{equation}\label{eqJS}
\card\{n \le N \colon n \equiv a (\mod 2^k3^{\l})\} \; \text{ and }\;  s_2(n)=s_3(n)\} \gg N^{\b}.
\end{equation}
\end{theorem}

We shall conclude the paper by some consideration on the question of characterizing the integers $b$ such that for any integer $a \in \{1, 2, \ldots, b-1\}$, there exist infinitly many integers for which $\l_b(n!)=a$.

\section{Acknowledgements}
Jean-Marc Deshouillers acknowledges support by the FWF--ANR project ArithRand (grant numbers I4945 and ANR-20-CE91-0006)
Pascal Jelinek was supported by the FWF project P36137 and the FWF--RSF joint project DIOARIANA (grant number I5554)
Lukas Spiegelhofer acknowledges support by the FWF--ANR projects ArithRand (grant number I4945), and SymDynAr (grant number I6750).

\section{Proof of Theorem \ref{thmlnzdigit} assuming Theorem \ref{thmJS}}\label{seclnz}

\subsection{Valuation and last nonzero digit}

Let $b$ be an integer larger than $1$ and $n$ a positive integer. We denote by $v_b(n)$ the largest nonnegative integer such that $b^{v_b(n)}$ divides $n$ and we write
\begin{equation}\label{val}
n = b^{v_b(n)} n_b, \; \text{ where $b$ does not divide } n_b.
\end{equation}
The following relations are easily verified
\begin{eqnarray}
\label{valprod}
\text{If $b_1$ and $b_2$ are coprime, then } v_{b_1b_2}(n) &= &\min\left(v_{b_1}(n), v_{b_2}(n)\right),\\
\label{valandlnzd}
\text{for any $b\ge 2$ } \colon \l_b(n)& \equiv& n_b (\mod b).
\end{eqnarray}
In the special case when $b=12$, we deduce the following
\begin{eqnarray}
\label{v4large}
v_4(n) > v_3(n)& \Longleftrightarrow& \l_{12}(n) \in \{4, 8\},\\
\label{v3large}
v_4(n) < v_3(n)& \Longleftrightarrow& \l_{12}(n) \in \{3, 6, 9\},\\
\label{v3isv4}
v_4(n) = v_3(n) &\Longleftrightarrow& \l_{12}(n) \in \{1,2,5,7,10,11\}.
\end{eqnarray}

We owe to Legendre a compact formula for the the vauation of $n!$ in a prime base $p$, namely
\begin{equation}\label{Legendre}
v_p(n!) = \frac{n-s_p(n)}{p-1}, \; \text{ where the function $s$ is defined in (\ref{sumdig}).}
\end{equation}
This implies that for a prime power base $p^{a}$ we have
\begin{equation}\label{Legendregen}
v_{p^{a}}(n!) = \left\lfloor\frac{n-s_p(n)}{a(p-1)}\right\rfloor.
\end{equation}

In the special case of the base $12$, we have
$$
v_3(n!) = \frac{n-s_3(n)}{2} \; \text{ and } \; v_4(n!) = \left\lfloor \frac{n-s_2(n)}{2} \right\rfloor
$$
and thus, by (\ref{v3isv4}), we have
\begin{equation}\label{exceptcase}
\l_{12}(n!) \in \{1,2,5,7,10,11\} \Longleftrightarrow s_3(n) \in\{s_2(n), s_2(n)+1\}.
\end{equation}
We can go one step further and show that
\begin{eqnarray}
\label{157eleven}
\l_{12}(n!) \in \{1,5,7,11\} &\Longleftrightarrow & s_3(n) = s_2(n)\\
\notag
 \text{ and }&& \\
 \label{2ten}
 \l_{12}(n!) \in \{2, 10\}& \Longleftrightarrow &s_3(n) = s_2(n)+1.
\end{eqnarray}
Let us prove it. We first assume that $s_2(n)=s_3(n)$; by Legendre's formula (\ref{Legendre}), $n!$ can be written as $3^{u}R$ and $4^{u}S$, where $R$ is not divisible by $3$ and $S$ is not divisible by $2$; thus $n! = 12^{u}T$ where $T$ is not divisible by $2$ or $3$;  thus  $\ell_{12}(n!)$ belongs to $\{1, 5, 7, 11\}$; this implies the ``$\Leftarrow$ implication" in (\ref{157eleven}). In a similar way, when $s_2(n) = s_3(n)-1$, we can write $n! = 12^{u}T$ for some $T$ which is divisble by $2$ but neither by $4$ nor $3$;  thus  $\ell_{12}(n!)$ belongs to $\{2,10\}$, which implies the ``$\Leftarrow$ implication" in (\ref{2ten}). Those two relations, combined with the equivalence (\ref{exceptcase}), imply the equivalence in (\ref{157eleven}) and in (\ref{2ten}).

\subsection{Reduction to consecutive blocks}

From (\ref{exceptcase}), the main result of \cite{S} implies that at least one of the digits $1, 2, 5, 7, 10, 11$ occurs quite often, but it may be difficult to decide which are the digits which occur indeed sufficiently often. To show that indeed all of them occur sufficiently often, we shall use a trick introduced in \cite{D1} and consider blocks of consecutive integers in which all those values occur.
\begin{proposition}\label{consblocks}
If $n$ is divisible by $12^3$ and $s_2(n)=s_3(n)$, then
$$
\{1,2,5,7,10, 11\} \subset \{\ell_{12}((n+1)!), \ell_{12}((n+2)!), \dots, \ell_{12}((n+10)!))\}.
$$
\end{proposition}
\begin{proof}
	The statement is a simple consequence of the following claims.
	\begin{claim}\label{digitfact}
		We have
		\begin{equation}\notag
			(\ell_{12}(1!), \ell_{12}(2!), \ell_{12}(5!), \ell_{12}(6!), \ell_{12}(7!), \ell_{12}(10!)) = (1, 2, 10, 5, 11, 7).
		\end{equation}
		\begin{proof}
			This is a simple computation; one can also look at the sequence A136698 in \cite{OEIS}.
		\end{proof}
	\end{claim}
	\begin{claim}\label{prodtofact}
		If $n$ is divisible by $12^3$ and $k$ is an integer in $[1, 10]$, we have
		\begin{equation}\label{digconsprod}
			\ell_{12}((n+1)(n+2)\cdots (n+k)) = \ell_{12}(k!).
		\end{equation}
	\end{claim}
	\begin{proof}
		Since $n$ is divisible by $12^3$, we have
		$$
		(n+1)(n+2)\cdots(n+k) \equiv k! (\operatorname{mod} 12^3).
		$$
		For $k \le 8$, the claim follows from the fact that $8!$ is not divisible by $12^3$. 
		
		For $k$ in $\{9, 10\}$, we go one step further in the 
		expension of the product as a polynomial in $n$, keeping the first degree terms in $n$, i.e.
		$$
		(n+1)(n+2)\cdots(n+k)  \equiv k! + n\left(\frac{k!}{1} + \cdots + \frac{k!}{k} \right) (\operatorname{mod} 12^6).
		$$
		We notice that $k!$ is not divisible by $12^5$ but all the terms inside the bracket in the RHS are divisible by $12^2$.
	\end{proof}
	\begin{claim}\label{multinv}
		Let $x$ and $y$ be two integers such that $\ell_{12}(x)$ is in $\{1, 5, 7, 11\}$, we have
		\begin{equation}
			\ell_{12}(xy) \equiv \ell_{12}(x) \ell_{12}(y) (\operatorname{mod} 12).
		\end{equation}
	\end{claim}
	\begin{proof}
		We can write $x = 12^{\alpha}(\ell_{12}(x) + 12X)$ and $y=12^{\beta}(\ell_{12}(y)+12Y)$. Thus, we have
		$xy = 12^{\alpha+\beta}(\ell_{12}(x) \ell_{12}(y) +12Z)$. Since $\ell_{12}(y) $ is not zero and $\ell_{12}(x) $ is invertible modulo $12$, the product $\ell_{12}(x) \ell_{12}(y) $ is not congruent to $0$ modulo $12$ and it is thus congruent to $\ell_{12}(xy)$ modulo $12$. 
	\end{proof}
	\begin{claim}\label{ensinvar}
		The set  $\{1, 2, 10, 5, 7, 11\}$ of residues modulo $12$ is invariant by multiplication by $1$, $5$, $7$ or $11$.
	\end{claim}
	\begin{proof}
		A simple computation shows that we respectively get the sets $\{1, 2, 10, 5, 7, 11\}$, $\{5, 10, 2, 1, 11, 7\}$, $\{7, 2, 10, 11, 1, 5\}$ and $\{11, 10, 2, 7, 5, 1\}$.
	\end{proof}
	Let $n$ satisfy the conditions of Proposition \ref{consblocks} and $k$ be in $\{1, 2, 5, 6, 7, 10\}$. By (\ref{157eleven}), $\ell_{12}(n!)$ is in $\{1, 5, 7 , 11\}$. By Claim \ref{multinv}, we have $\ell_{12}((n+k)!) \equiv \ell_{12}(n!) \ell_{12}((n+1)(n+2)\cdots (n+k)) (\operatorname{mod} 12)$. By Claim \ref{prodtofact},  we have \\
	${\ell_{12}((n+k)!) \equiv
		\ell_{12}(n!) \ell_{12}(k!) (\operatorname{mod} 12)}$. By (\ref{157eleven}) and Claim \ref{digitfact}, the set \\
	${(\ell_{12}((n+1)!), \ell_{12}((n+2)!), \ell_{12}((n+5)!), \ell_{12}((n+6)!), \ell_{12}((n+7)!), \ell_{12}((n+10)!))}$ is - modulo $12$ - the product by $1$, $5$, $7$ or $11$ of the set $\{1, 2, 10, 5, 7, 11\}$, and by Claim \ref{multinv}, it is the set $\{1, 2, 5, 7, 10, 11\}$. This ends the proof of Proposition \ref{consblocks}.
	\end{proof}

\subsection{Proof of Theorem \ref{thmlnzdigit}}\label{subseclnz}

We start with the lower bound. Since we are concerned with finitely many values of $a$, it is enough to prove that for each value of $a$ there exists $\delta$ such that the lower bound of (\ref{eqlnzdigit}) holds. 

For $a$ in $\{1, 2, 5, 7, 10, 11\}$, this is a simple consequence of Proposition \ref{consblocks} and Theorem \ref{thmJS} applied with $a=0, k=6$ and $\l = 3$.

Let us now assume that $a$ is in $\{3, 6, 9\}$. The following results (Proposition 3.) was proved in \cite{D1}
\begin{proposition}\label{Prop3}
Let $n$ be divisible by $144$ and be such that $v_3(n!) \ge v_4(n!)+2$. Then for any $a \in \{3, 6, 9\}$ there exists $k \in \{0, 2, 3, 7\}$ such that $\ell_{12}((n+k)!)=a$.
\end{proposition}
The number $R=3^92^4$ was introduced in \cite{D1}; it satisfies the following properties
$144 \| R, s_2(R)=8, s_3(R) = 4, R < 2^{19} \; \text{ and }\; R< 3^{12}$. \\
By Theorem \ref{thmJS}, there are more than $N^{\b}$ integers  $n \le N$ which are divisible by $2^{19}3^{12}$ and such that $s_2(n)=s_3(n)$.\\
For such $n$, one has $s_2(n+R) = s_2(n) + 8$ and $s_3(n+R)=s_3(n)+4=s_2(n)+4 = s_2(n+R)-4$. By Legendre's formula, we thus have $v_3((n+R)!) \ge v_4((n+R)!)$. By contruction, $n+R$ is divisible by $144$ and by Proposition \ref{Prop3}, all the digits $3, 6$ and $9$ occur in the sequence $\{\ell_{12}((n+R)!), \ell_{12}((n+R+2)!), \ell_{12}((n+R+3)!), \ell_{12}((n+R+7)!)\}$. \\

As regards the upper bound, the following was stated in \cite{DR} 

\begin{claim}\label{fewoutside48}
\textit{We can show that $\card\{n \le x \colon \l_{12}(n!)=a\} =O(x^c)$ with some $c<1$ whenever $a \neq 4, 8$.}'
\end{claim}
We give here a proof of this claim, based on the following \textit{folklore} result, for which we give a short proof relying on Hoeffding's inequality.\\
\begin{lemma}\label{folklore}
Let $b \ge 2$ be an integer and  $\delta$  a positive real number. There exists a real number $c < 1$ such that for all positive real $x$ one has
\begin{equation}\label{farfromthemean}
\card\left\{0 <n < x \colon \left|s_b(n)-\frac{b-1}{2\log b}\log n\right|\ge \delta \log n\right\} \ll x^c.
\end{equation}
\end{lemma}
\begin{proof}
It is enough to prove the lemma when $x$ is a sufficiently large power of $b$, say $x=b^K$. A little computation shows that we can find $\lambda$ and $\mu$ with $0 < \lambda< \mu<1$ such that for any $k$ in $[\lambda K, \mu K]$ one has
$$
(K-k)\left(\delta \log b + (b-1)/2\right) \le \delta / 2 \log b^K.
$$
Then, for $n$ in $[b^k, b^K)$, one has
\begin{equation}\label{reductiontolargen}
\left|s_b(n)-\frac{b-1}{2\log b}\log n\right|\ge \delta \log n  \; \Rightarrow \; \left|s_b(n)-\frac{b-1}{2}K\right|\ge (\delta/2)\log b \times K.
\end{equation}

Since there are $(b^K)^{k/K}$ integers less than $b^k$ and $k/K \le \mu < 1$, Equation (\ref{reductiontolargen}) implies that to prove Lemma \ref{farfromthemean}, it is enough to prove that for any positive $\eta$, one has for any positive $\eta$ the existence of a real number $c <1$ such that 
\begin{equation}\label{reductocube}
 \card\left\{0 \le n < b^K   \colon \left|s_b(n)-\frac{b-1}{2}K\right|\ge \eta K\right\} \ll b^{cK}, \; \text{ as $K \rightarrow \infty$}.
\end{equation}
We can rephrase this question in a probabilistic setting. Let $\xi_1, \cdots, \xi_K$ be $K$ independent random variables uniformly distributed on $\{0, 1, \cdots, b-2, b-1\}$ and let $S_K = \xi_1 + \ldots +\xi_K$. We readily see that $S_K$ has the following properties
\begin{eqnarray}
\label{distrib}
\text{For any integer $ m $}\colon \P\{S_K=m\} &=& b^{-K} \card\{n < b^K \colon s_b(n)=m\},\\
\label{mean}
\mathbb{E}(S_K) &=& K\frac{b-1}{2}.
\end{eqnarray}
Hoeffding's inequality \cite{H} gives the following upper bound for the tail of the distribution of $S_K$
\begin{equation}\label{Hoeffding}
\P\left\{\left|S_K-\mathbb{E}(S_K)\right| \ge t \right\} \le 2 \exp\left(- \frac{2 t^2}{K(b-1)^2}\right).
\end{equation}
Relations (\ref{distrib}), (\ref{mean}) and (\ref{Hoeffding}) imply
\begin{eqnarray*}
 \card\left\{0 \le n < b^K   \colon \left|s_b(n)-\frac{b-1}{2}K\right|\ge \eta K\right\} &=& b^K \P\left\{\left|S_K-\mathbb{E}(S_K)\right| \ge \eta K \right\} \\
& \le &2 b^K \exp\left(-\frac{2(\eta K)^2}{K(b-1)^2}\right)\\
&\le& 2 b^{\left(1-\left(\frac{2 \eta^2}{(b-1)^2\log b}\right)\right)K},
\end{eqnarray*}
which proves the validity of (\ref{reductocube}) and thus ends the proof of Lemma \ref{folklore}.\\

By (\ref{v4large}), $\l_{12}(n!) \notin \{4,8\}$ is equivalent to $v_4(n!) \le v_3(n!)$ and thus
(\ref{Legendre}) implies that $\l_{12}(n!) \notin \{4,8\}$ is equivalent to $s_3(n) \le s_2(n)+1$. This relation occurs only rarely: 
\begin{itemize}
\item either $s_3(n) \le 0.82 \log n$. Since $(3-1)/(2\log 3) = 0.91...$, the set of such integers $n$ has an exponential density less than $1$ by Lemma \ref{folklore}, applied with $b=3$,
\item or $s_3(n) > 0.82 \log n$. We have then $s_2(n) \ge s_3(n) -1 > 0.82 \log n -1$ . Since $(2-1)/(2\log 2) = 0.72...$, the set of such integers $n$ has an exponential density less than $1$ by Lemma \ref{folklore}, applied with $b=2$.
\end{itemize}
This ends the proof of Claim \ref{fewoutside48}, as well as that of Theorem \ref{thmlnzdigit}.

\end{proof}

\section{Proof of Theorem \ref{thmJS}}\label{secJS}

\subsection{Revisiting the collisions-paper}
In this section, we indicate how to obtain a complete proof of Theorem~\ref{thmJS}.
This is accomplished by minimalistic changes of the paper~\cite{S}, which we indicate in the sequel.

First, we want to find \emph{almost-collisions} in a given residue class $a+2^k3^\ell\mathbb N$, replacing the statement ``$\digitsum_2(n)-\digitsum_3(n)\in\{0,1\}$ for infinitely many $n\in9+12\mathbb N$''.
For this, we modify the definition of $\eta$ (\cite[Equation~(12)]{S}) by choosing $m=\max\{2,k\}$, and $\eta\coloneqq2^m\bigl\lfloor \eta_02^{-m}\bigr\rfloor$.


In~\cite[Proposition~2.2]{S},
the definition of the shifts $d_j$ has to be changed in order to ensure that we do not leave our prescribed residue class modulo $2^k3^\ell$.
For this, let
\[d_j\coloneqq \bigl(1^{(j+1+J)\eta}\tO^\ell\bigr)_3,\]
which is obviously divisible by $3^\ell$.
Moreover, the fact
\[
2^m\mid\sum_{0\leq j<R\cdot2^m}3^j
\]
(use a geometric series and Euler--Fermat)
implies $2^k\mid d_j$.
Moreover, we ask for the existence of $L\in\{0,\ldots,2^\nu3^\beta-1\}$ such that $L\equiv a\bmod 2^k3^\ell\mathbb N$, instead of $L\equiv 9\bmod 12$.
The rest of the statement is unchanged.
In the proof, ``blocks of $\tL$s of length $\eta$'' appear.
Their length is a multiple of four also in the present paper, and the three addition patterns depicted below~\cite[Equation~(42)]{S} can be reused in order to obtain an ``almost-collision'' in the sense $\digitsum_2(n)-\digitsum_3(n)\in\{0,1\}$.

In~\cite[Equation~(35)]{S}, we replace the first line by a congruence $a\equiv r\bmod 2^m$, where $r$ is arbitrary.
Also, the construction of the integers $k_j$, beginning with~\cite[Equation~(43)]{S}, has to account for an arbitrary residue class modulo $3^\ell$.
This flexibility is needed for obtaining collisions in \emph{any} residue class modulo $2^k3^\ell$:
by intersection of residue classes such as in~\cite[Equation~(41)]{S}, and the Chinese remainder theorem, we can obtain any prescribed $L$.

Reusing~\cite[Propositions~2.3 and~2.4]{S}, we obtain the following statement.

\begin{proposition}\label{prp_almost_collisions}
Let $k,\ell\ge0$ be integers, and $0\leq L<2^k3^\ell$. There exist infinitely many positive integers $n\in L+2^k3^\ell\mathbb N$ such that
\[\digitsum_2(n)-\digitsum_3(n)\in\{0,1\}.\]
is bounded below by $C$.
\end{proposition}

\subsection{Removal of almost collisions}
In the previous subsection we have shown that we can achieve almost collisions in any residue class $L \bmod 2^k3^{\ell}$. We now want to remove the inelegant deviation by $1$ from a true collision, while staying in the residue class. The following proposition states that we can move to a larger residue class $L' \bmod 2^{k'}3^{\ell'}$ to change an almost collision into a collision in the smaller residue class, hence proving Theorem \ref{thmJS}. 
\begin{proposition}\label{Almost}
    Let $2\leq p<q$ be two coprime integers. Let $a_0$ be a non-negative integer, say it is smaller than $p^{y_{p_0}}q^{y_{q_0}}$, for some integers $p_0$ and $q_0$. \\
    Let $d_i, m_i, M_i$ each be a collection of $\ell$ integers. Then the following 2 statements are equivalent:
    \begin{enumerate}
        \item each $d_i$ is divisible by $\gcd(p-1,q-1)$;
        \item There exist integers $t_1,\dots,t_{\ell}, y_p>y_{p_0}, y_q>y_{q_0}$ and $a$ which satisfy:
        \begin{itemize}
            \item $t_i\equiv m_i \bmod M_i$ for all $1\leq i \leq \ell$
            \item $a\equiv a_0 \bmod p^{y_{p_0}}q^{y_{q_0}}$
        \end{itemize}
        and for all $n \equiv a \bmod p^{y_p}q^{y_q}$ and for all $1\leq i \leq \ell$ the following holds:
	\begin{equation}\label{egn_shift}
	    \digitsum_p(n+t_i)-\digitsum_q(n+t_i)= \digitsum_p(n)-\digitsum_q(n)+d_i \,.
	\end{equation}
    \end{enumerate}
\end{proposition}
\noindent
Before we can prove this proposition, we need the following lemma.
\begin{lemma} \label{trick}
	Let $c$ be the highest nonzero coefficient in the $p$-adic expansion of a positive integer $m$, say it is the coefficient of $p^y$. Further, assume that $c\neq p-1$. If we add a positive integer $n$ satisfying $n\equiv cp^{y+z}-cp^y + r\bmod p^{y+z+1}$, for some integers $z>0, 0\leq r < p^y$, then the following equation holds:
	$$
	\digitsum_p(n+m)=\digitsum_p(m+r)+\digitsum_p(n-r)- z(p-1)\,.
	$$
    In the case $c=p-1$, we would get that
    $$
	\digitsum_p(n+m)=\digitsum_p(m+r)+\digitsum_p(n-r)- (z-1)(p-1)\,.
	$$
\end{lemma}
\begin{proof}
    This lemma becomes clear once we write out the numbers in question in base $p$. 
    We consider the addition of the following two numbers:
    $$\begin{array}{rccccccc}
         m:\;&0&0&0& \cdots&0&c&m' \\
         n:\;&n'&\mathrel{{c}{-}{1}}&\mathrel{{p}{-}{1}}&\cdots&\mathrel{{p}{-}{1}}&\mathrel{{p}{-}{c}}&r\\
         \hline
         n+m:\;&n'& c & 0 & \cdots & 0 & \multicolumn{2}{c}{m'+r}
    \end{array}
    $$
    Hence we see that the last $y+z+1$ digits contribute $\digitsum_p(m+r)$ in the case that $c\neq p-1$ and contribute $\digitsum_p(m+r)+p-1$ in the case that $c= p-1$, while the remaining digits contribute $\digitsum_p(n')=\digitsum_p(n-r-cp^{y+z}+cp^y)=\digitsum_p(n-r)-z(p-1)$.
\end{proof}
\noindent
Now we will be able to prove Proposition \ref{Almost}.
\begin{proof}
    First, we will show that $(2)$ implies $(1)$. It is known that $\digitsum_p(n) \equiv n \bmod p-1$, and $\digitsum_q(n) \equiv n \bmod q-1$. Hence evaluating equation \eqref{egn_shift} modulo $\gcd(p-1,q-1)$ we get:
    \begin{equation*}
        n+t_i-(n+t_i)\equiv n-n + d_i \;.
    \end{equation*}
    Hence $d_i \equiv 0 \bmod \gcd(p-1,q-1)$.
    
    Now we will prove that $(1)$ implies $(2)$.
    We will give an recursive algorithm that gives us $t_1$ up to $t_{\ell}$, and the corresponding residue class $a$ given any $a_0$.
    Without loss of generality, we can assume that $a_0<p^{y_{p_0}}q^{y_{q_0}}$. Pick $t_1>p^{y_{p_0}+1}q^{y_{q_0}+1}$ such that $t_1\equiv m_1 \bmod M_1$. By Bezout's lemma we can find $z'_{p_1}$ and $z'_{q_1}$ integers such that:
	\begin{equation*}
		-d_{1}+\digitsum_p(t_{1}+a_{0})-\digitsum_q(t_{1}+a_{0})-\digitsum_p(a_{0})+\digitsum_q(a_{0})=z'_{p_1}(p-1)-z'_{q_1}(q-1)\,.
	\end{equation*}
    Let $c_{p_1}$ be the leading digit of $t_{1}$ in base $p$, say it is the coefficient of $p^{\tilde y_{p_1}}$, and let $c_{q_1}$ be the leading digit of $t_{1}$ in base $q$, say it is the coefficient of $q^{\tilde y_{q_1}}$. Additionally, let $z_{p_1}=z'_{p_1}$, if $c_{p_1}\neq p-1$, and $z_{p_i}=z'_{p_1}+1$, if $c_{p_1}=p-1$. Define $z_{q_1}$ analogously.\\
	Now pick $n$ such that it satisfies:
	\begin{align}
		n &\equiv c_{p_1}p^{\tilde y_{p_1}+z_{p_1}}-c_{p_1}p^{\tilde y_{p_1}} + a_{0} \bmod p^{\tilde y_{p_1} + z_{p_1}+1}\\
		n &\equiv c_{q_1}q^{\tilde y_{q_1}+z_{q_1}}-c_{q_1}q^{\tilde y_{q_1}} + a_{0} \bmod q^{\tilde y_{q_1} + z_{q_1}+1}
	\end{align} 
	Let the residue class of $n$ $\bmod$ $p^{\tilde y_{p_1} + z_{p_1}+1}q^{\tilde y_{q_1} + z_{q_1}+1}$ be called $a_1$. Also denote \newline $y_{p_1}=\tilde y_{p_1}+z_{p_1}+1$ and $y_{q_1}=\tilde y_{q_1}+z_{q_1}+1$\\
	Using Lemma \ref{trick} twice, once with $c=c_{p_1}, y=\tilde y_{p_1}, r=a_{0}, z=z_{p_1}$ and once with $c=c_{q_1}, y=\tilde y_{q_1}, r=a_{0}, z=z_{q_1}$, we have that:
    \begin{equation}
        \begin{aligned}
		\digitsum_p(n + t_{1})&=\digitsum_p(t_{1}+a_{0})+\digitsum_p(n-a_{0})-z'_{p_1}(p-1)\\
		\digitsum_q(n + t_{1})&=\digitsum_q(t_{1}+a_{0})+\digitsum_q(n-a_{0})-z'_{q_1}(q-1)\, .
	\end{aligned}
    \end{equation}
    
	Subtracting the second line from the first, we get by the choices of $z'_{p_1}$ and $z'_{q_1}$:
    \begin{align}
		\digitsum_p(n + t_{1})-\digitsum_q(n + t_{1})&=\digitsum_p(n-a_{0})+\digitsum_p(a_{0})-\digitsum_q(n-a_{0})-\digitsum_q(a_{0})+d_{1}\\
        \digitsum_p(n + t_{1})-\digitsum_q(n + t_{1})&=\digitsum_p(n)-\digitsum_q(n)+d_{1}\, ,
	\end{align} 
    since $a_{0}$ are the last $y_p$ or $y_q$ digits of $n$, respectively. \\
    Now we will repeat this process another $\ell-1$ times, and during each iteration, we increase the indices of each variable by 1 compared to the previous iteration.
    Since in the $k$-th iteration, we have not changed the last $y_{p_{k-1}}$ or $y_{q_{k-1}}$ digits of $n$, respectively, $n$ is still in the same residue class $\bmod$ $p^{y_{p_{k-1}}}q^{y_{q_{k-1}}}$ and hence $t_1,\dots,t_{k-1}$ satisfy still their equation respectively for $n$ given in the new residue class. Taking $a$ to be $a_{\ell}$ completes the proof.
\end{proof}

\section{A general question}\label{secMoh}

By the main result of \cite{DR} and Theorem \ref{thmlnzdigit}, the base $b=12$ has the following property
\begin{property}[Full-range condition]\label{fullrange}
$\forall a \in \{1, 2, \ldots, b-1\}, \exists_{\infty} n \colon \l_b(n!) =a$.
\end{property}
Thanks to B. Sobolewski \cite{So}, we know that all bases that are prime powers also satisfy Property~\ref{fullrange}.
It would be interesting to characterize all integers $b$ satisfying the \textit{full-range condition}. \\

The following straightforward claim leads to a necessary condition for a base $b$ to satisfy Property~\ref{fullrange}.

\begin{claim}\label{cnecesall}
Let $b$ be a natural integer larger than $2$ and $b=p_1^{a_1}\cdots p_s^{a_s}$ a decomposition of $b$ as a product  of pairwise coprime prime powers. A necessary condition that at least one integer $a$ coprime with $b$ occurs  infinitely many times in the sequence $( \l_b(n!))_n$ is 
\begin{equation}\label{enec}
a_1(p_1-1) = \cdots = a_s(p_s-1).
\end{equation}
\end{claim}
\begin{proof}
If the condition (\ref{enec}) is not satisfied, then $s \ge 2$ and  there exist two indices $i$ and $j$ such that $a_i(p_i-1)<a_j(p_j-1)$. By (\ref{Legendregen}) and the fact that $s_p(n) = O(\log n)$, we have, for sufficiently large $n$ : $v_{p^{j}}(n!) < v_{p^{i}}(n!)$ and so $p^{i}$ divides $\l_b(n!)$.
\end{proof}

By the same reasoning, we easily prove the following
\begin{claim}\label{cnecesuffall}
Let $b$ be a natural integer larger than $2$ and $b=p_1^{a_1}\cdots p_s^{a_s}$ a decomposition of $b$ as a product  of pairwise coprime prime powers. A necessary and sufficient condition that at least one integer $a$ coprime with $b$ occurs  infinitely many times in the sequence $( \l_b(n!))_n$ is 
\begin{equation}\label{enecsuf}
a_1(p_1-1) = \cdots = a_s(p_s-1) \; \text{ and }\; \exists_{\infty} n \colon s_{p_{1}}(n) = \cdots =s_{p_{s}}(n).
\end{equation}
\end{claim}

We believe that Claim \ref{cnecesuffall} is also necessary and sufficient for $b$ to satisfy the \textit{full-range condition} (Property~\ref{fullrange}).

In a letter to the Number Theory List \cite{Sun}, Zhi-Wei Sun mentions the following
\begin{conjecture}[Shawkwei Moh, 1990]\label{cmoh}
For any $k$ distinct primes $p_1,\ldots, p_k$, there are infinitely many positive integers $n$ such that
\begin{equation}\label{emoh}
v_{p_1}(n!) : v_{p_2}(n!) : \cdots : v_{p_k}(n!) = \frac{1}{p_1-1} :  \frac{1}{p_2-1} : \cdots :  \frac{1}{p_k-1} .
\end{equation}
\end{conjecture}
This conjecture would imply that Conditions (\ref{enec}) and (\ref{enecsuf}) are equivalent. We have however some heuristic reasons, which we shall present elsewhere, to doubt that this is the case.

\bigskip
\begin{center}
\begin{tabular}{c}
Institut de Math\'{e}matiques de Bordeaux,\\[0.5mm]
Universit\'{e} de Bordeaux, Bordeaux INP and CNRS,\\[0.5mm]
33405 Talence, France,\\[0.5mm]
jean-marc.deshouillers@math.u-bordeaux.fr,\\[0.5mm]
ORCID iD: \texttt{0000-0002-1826-689X}
\end{tabular}
\end{center}

\bigskip
\begin{center}
\begin{tabular}{c}
Montanuniversit\"at Leoben,\\[0.5mm]
Franz-Josef-Strasse 18, 8700 Leoben, Austria\\[0.5mm]
pascal.jelinek@unileoben.ac.at\\[0.5mm]
ORCID iD: \texttt{0009-0004-5877-2908}
\end{tabular}
\end{center}

\bigskip
\begin{center}
\begin{tabular}{c}
Montanuniversit\"at Leoben,\\[0.5mm]
Franz-Josef-Strasse 18, 8700 Leoben, Austria\\[0.5mm]
lukas.spiegelhofer@unileoben.ac.at\\[0.5mm]
ORCID iD: \texttt{0000-0003-3552-603X}
\end{tabular}
\end{center}


\begin{thebibliography}{}


\bibitem{D1} Deshouillers, Jean-Marc. -\textit{ A footnote to ``The least non zero digit of $n!$ in base $12$''}, Unif. Distrib. Theory, 7 (2012), 71-73.

\bibitem{DR} Deshouillers, Jean-Marc and Ruzsa, Imre. - \textit{The least nonzero digit of $n!$ in base $12$}, Publ. Math. Debrecen, 79 (2011), 163-167.

    

\bibitem{H} Hoeffding, Wassily. - \textit{Probability inequalities for sums of bounded random variables}, J. Amer. Statist. Assoc. 58 (1963), 13-30.

\bibitem{OEIS} Sloane, Neil J. A. - \textit{The On-Line Encyclopedia of Integer Sequences},  \verb"https://oeis.org".

\bibitem{So} Sobolewski, B .- \textit{On the last nonzero digits of $n!$ in a given base}, Acta Arith. 191 (2019), 173-189.

\bibitem{S} Spiegelhofer, Lukas. - \textit{Collisions of digit sums in bases $2$ and $3$}, Israel J. of Math. 258 (2023), 475-502.

\bibitem{Sun} Sun, Zhi-Wei. - \textit{A conjecture of Shawkwein Moh}, a letter sent to nmbrthry@listserv.nodak.edu on  March 11th, 2012.


\end{thebibliography}
\end{document}